\documentclass[11pt]{amsart}
\usepackage{graphicx}
\usepackage[active]{srcltx}
 \makeatletter
\subjclass[2010]
 \makeatother
\usepackage{enumerate,url,amssymb,  mathrsfs}

\newtheorem{theorem}{Theorem}[section]
\newtheorem{lemma}[theorem]{Lemma}
\newtheorem*{lemma*}{Lemma}
\newtheorem{proposition}[theorem]{Proposition}
\newtheorem{corollary}[theorem]{Corollary}

\theoremstyle{definition}

\theoremstyle{remark}
\newtheorem{remark}[theorem]{Remark}

\numberwithin{equation}{section}

\newcommand{\abs}[1]{\lvert#1\rvert}

\def\XXint#1#2#3{{\setbox0=\hbox{$#1{#2#3}{\int}$}
\vcenter{\hbox{$#2#3$}}\kern-.5\wd0}}

\def\le{\leqslant}
\def\ge{\geqslant}
\begin{document}

\title[Harmonic q.c. mappings between $\mathscr{C}^1$ smooth Jordan domains]{Harmonic quasiconformal mappings between $\mathscr{C}^1$ smooth Jordan domains}

\author{David Kalaj}
\address{Faculty of Natural Sciences and
Mathematics, University of Montenegro,  Cetinjski put b.b. 81000 Podgorica, Montenegro}
\email{davidkalaj@gmail.com}
\subjclass[2000]{Primary 30C62; Secondary 30C20, 31A20 }
\keywords{Harmonic mappings, quasiconformal mappings, smooth domains}

\begin{abstract}
 We prove the following result. If $f$ is a harmonic  quasiconformal mapping between two Jordan domains $D$ and  $\Omega$ having  $\mathscr{C}^1$ boundaries, then the function $f$ is globally H\"older continuous for every $\alpha<1$ but it is not necessarily Lipschitz in general. This result extends and improves a classical theorem of S. Warschawski for conformal mappings.
\end{abstract}

\maketitle

\section{Introduction}

Let $U$ and $V$ be two domains in the complex plane $\mathbf{C}$. We say that a twice differentiable mapping $f=u+iv:U\to V$ is harmonic if $\Delta f:=\Delta u + i\Delta v=0$ in $U$. Any harmonic homeomorphism is by Lewy's theorem a diffeomorphism. If its Jacobian $J_f=|f_z|^2-|f_{\bar z}|^2$ is positive, then it is a sense-preserving.

We say that a function $u:D\to \mathbf R$ is ACL (absolutely
continuous on lines) in  the region $D$, if for every closed
rectangle $R\subset D$ with sides parallel to the $x$ and $y$-axes,
$u$ is absolutely continuous on a.e. horizontal and a.e. vertical
line in $\mathbf R$. Such a function has partial
derivatives $u_x$, $u_y$ a.e. in $D$.

A sense-preserving homeomorphism  $w\colon D\to \Omega, $ where $D$
and $\Omega$ are subdomains of the complex plane $\mathbf C,$ is
said to be $K$-quasiconformal ($K$-q.c), { with} $K\ge 1$, if $w$ is
ACL in $D$ in the sense that its real and imaginary part are ACL in
D, and
\begin{equation}\label{defqc} |D w|\le K
l(D w)\ \ \ \text{a.e. on $D$},\end{equation} (cf. \cite{Ahl},
pp. 23--24).
Here $A=D(w)$ is the formal differential matrix defined by $$A=\left(
                                                                 \begin{array}{cc}
                                                                   u_x & u_y \\
                                                                   v_x & v_y \\
                                                                 \end{array}
                                                               \right),$$ and $$|A|=\max_{|h|=1}|Ah|,\ \ \ l(A)=\min_{|h|=1}|Ah|,$$ where $|\cdot|$ is the Euclidean norm.
Notice that the condition (\ref{defqc}) can be written
as
$$|w_{\bar z}|\le k|w_z|\quad \text{a.e. on $D$ where
$k=\frac{K-1}{K+1}$ i.e. $K=\frac{1+k}{1-k}$ }.$$

The class of quasiconformal harmonic mappings has been firstly considered by O. Martio in  \cite{martio}. The class of q.c. harmonic mappings contains conformal mappings, and this is why the class has shown a large interest for experts in geometric function theory.

We should mention here the following result of Pavlovi\'c \cite{MP} which states that a harmonic quasiconformal mapping of the unit disk $\mathbf{D}$ onto itself is bi-Lipschitz continuous. In order to explain the importance of his result let us state the following two separate results. If we assume that the mapping $f:\mathbf{D}\to \mathbf{D}$ is merely quasiconformal, then it is only H\"older continuous with the H\"older coefficient $\alpha = \frac{1-k}{1+k}$. This is the celebrated Mori's theorem.  On the other hand, if $f:\mathbf{D}\to \mathbf{D}$ is merely a harmonic diffeomorphism, then by a result of Hengartner and Schober it has  a continuous extension up to the boundary (see \cite[Theorem~4.3]{hs3} or \cite[Sec.~3.3]{duren}). However, in view of Rad\'o-Kneser-Choquet theorem, this is the best regularity that such a mapping can have at the boundary.

 We define the Poisson kernel by $$P(z,\theta)=\frac{1}{2\pi}\frac{1-|z|^2}{|z-e^{i\theta}|^2},\ \ \ |z|<1, \theta \in [0,2\pi).$$  For a mapping $f\in L^1(\mathbf{T})$, where $\mathbf{T}$ is the unit circle, the Poisson integral is defined by $$w(z) = P[f](z) =\int_{0}^{2\pi} P(z,\theta) f(e^{i\theta})d\theta.$$ The well-known Rad\'o-Kneser-Choquet theorem  states. If $f$ is a homeomorphism of the unit circle onto a convex Jordan curve $\gamma$, then its Poisson integral is a harmonic diffeomorphism of the unit disk $\mathbf{D}$ onto the Jordan domain $\Omega$ bounded by $\gamma$.

A special case is when $\gamma=\mathbf{T}$. E. Heinz has proved that, if $f$ is a harmonic diffeomorphism  of the unit disk onto itself, then the Hilbert-Schmidt norm of its derivative: \begin{equation}\label{deri}\|Df\|^2=|f_x|^2+|f_y|^2\ge c,\end{equation} where $c>0$ depends only on $f(0)$. It follows from \eqref{deri}, that the inverse of a quasiconformal harmonic mapping of the unit disk onto itself is Lipschitz continuous. So the main achievement of Pavlovi\'c in \cite{MP} (see also \cite{lib}), was to prove that a harmonic quasiconformal mapping of the unit disk onto itself is Lipschitz continuous on the closure of the domain.

In order to formulate some additional results in this topic recall that a rectifiable Jordan curve is $\mathscr{C}$, Dini smooth, $\mathscr{C}^{1,\alpha}$, for $\alpha\in(0,1]$ if its arch-length parametrisation $g:[0,|\gamma|]\to \gamma$ is $\mathscr{C}^1$, Dini smooth and $\mathscr{C}^{1,\alpha}$ respectively. Here $|\gamma|$ is the length of $\gamma$.

In \cite{kamz}, the author proved that, every quasiconformal harmonic mapping between Jordan domains with $\mathscr{C}^{1,\alpha}$ boundaries is Lipschitz continuous on the closure of domain. Later this result has been extended to Jordan domains with only Dini smooth boundaries \cite{pacific}.

A bi-Lipschitz property for harmonic quasiconformal mappings of the half-plane onto itself has been established by the author and Pavlovi\'c in \cite{plane}.

Further it has been  shown in \cite{pisa} that a quasiconformal harmonic mappings between $\mathscr{C}^{1,1}$ (not-necessarily convex) Jordan domains is bi-Lipschitz continuous. The same conclusion is obtained in \cite{abu} by Bo\v zin and Mateljevi\'c for the case of  $\mathscr{C}^{1,\alpha}$ Jordan domains. Further results in two dimensional case can be found in \cite{newj}. Some results concerning the  several-dimensional case can be found in \cite{astala}, \cite{kalajarsen} and \cite{matvuo}. For a different setting concerning the class of quasiconformal harmonic mappings we refer to the papers \cite{chen,vesna, zhu}. For example the article  \cite{vesna}  deals with the following problem of the class of quasiconformal harmonic mappings.

The quasi-hyperbolic metric $d_h$ in a domain $D$ of complex plane is defined as follows. For each $z_1, z_2\in D$, $$d_h(z_1,z_2)=\inf\int_{\gamma} d(z,\partial D)^{-1}|dz|,$$ where the infimum is taken over all rectifiable arcs $\gamma$ joining $x_1$ and $x_2$ in $D$. V. Manojlovi\'c in \cite{vesna} proved the following theorem: if $f:D\to D'$ is a quasiconformal and harmonic mapping, then it is bi-Lipschitz with respect to quasihyperbolic metrics on $D$ and $D'$.

In order to formulate the main theorem of this paper let us define the chord-arc curves.  A  rectifiable Jordan curve  $\gamma$  is a $B-$chord-arc curve if $L_\gamma(z_1, z_2) \le
B|z_1 - z_2|$ for all $z_1$, $z_2 \in  \gamma$, where $L_\gamma(z_1, z_2)$ denotes the length of the shortest arc of $\gamma$
joining $z_1$ and $z_2$. Here $B\ge 1$.
\begin{theorem}\label{teoml}
 Let $D$ and $\Omega$ be  Jordan domains  having  $\mathscr{C}^1$ boundaries and assume that $a\in D$ and $b\in \Omega$. Assume that $\omega_D$ ($\omega_\Omega$) is the modulus of continuity of the derivative of arc-length parametrisation of $\partial D$ ($\partial\Omega$). Assume further that $\partial D$ and $\partial \Omega$ satisfy $B-$arc-chord condition for some $B\ge 1$.  Then for every $\alpha\in (0,1)$ and $k\in[0,1)$, there  is a constant $M_\alpha=M_\alpha(a,b,k,B, \omega_D,\omega_\Omega)$ so that every harmonic $K=(1+k)/(1-k)-$quasiconformal mapping $f=g+\bar h$ of $D$ onto $\Omega$ so that $f(a)=b$ satisfy the condition  \begin{equation}\label{holderw}|f(z) -f(w)|\le M_\alpha|z-w|^\alpha,\ \ z,w\in D.\end{equation} Moreover for every $p>0$, there is a constant $B_p$, that depends on the same parameters as $M_\alpha$ so that \begin{equation}\label{hardy}\int_{D}|Df(z)|^{ p}\mathrm{d}\lambda(z)\le B^p_p,\end{equation} where $|Df(z)|=|f_z|+|f_{\bar z}|=|g'|+|h'|$. In other words $g', h'$ belong to the Bergman space $\mathcal{A}^p$ for every $p>0$. Here $\lambda$ is the Legesgue's measure in the plane.
\end{theorem}

\begin{remark}
In  Theorem~\ref{teoml} we consider the mappings between Jordan domains. The same conclusion can be made for multiply-connected domain bounded by finite number of $\mathscr{C}^1$ Jordan curves. We also expect that a similar conclusion can be made for non-bounded domains, but we did not pursue this question seriously.
\end{remark}

\subsection{The organization of the paper} We continue this section with some immediate corollaries of the main result. We prove that a $K-$quasiconformal mapping between $\mathscr{C}^1$ domains is $\beta-$H\"older continuous for every $\beta<1/K$. In particular we prove that a conformal mapping is $\beta-$H\"older continuous for every $\beta<1$. In the second section we prove a variation of the main result which will be needed to prove to prove Theorem~\ref{teoml} in the full generality. The proof of Theorem~\ref{teoml} is presented in the last section. The proof depends on a two-side connection between the $\alpha-$H\"older constant and the so-called $\alpha-$Bloch type norm of the holomorphic function defined on the unit disk expressed in Lemma~\ref{lemba}.  By using this connection, and  by a subtle application of $\mathscr{C}^1$ smoothness of the boundary curve of the image domain, we first find an a priori estimate of the $\alpha-$H\"older constant of a harmonic quasiconformal mapping of the unit disk onto a $\mathscr{C}^1$ Jordan domain having $\mathscr{C}^1$ extension up to the boundary. Then we use an approximation argument to get an estimate of $\alpha-$H\"older constant for a harmonic q.c. mapping which has not necessary smooth extension up to the boundary. To deal with the mappings whose domain is not the unit disk is a simple matter having proved the results from the second section.

\subsection{Some immediate consequences}
\begin{corollary}\label{rrje1}\cite{les}
If $f$ is a univalent conformal mapping between two Jordan domains $D$ and $\Omega$ with $\mathscr{C}^1$ boundaries, then $f$ is $\alpha$ H\"older continuous for every $0<\alpha<1$. Moreover,  if $\partial D$ and $\partial \Omega$ satisfy $B-$arc-chord condition for some $B\ge 1$, then  for every $\alpha\in(0,1)$ and every $a\in D$ and $b=f(a)\in \Omega$,  there exists $M=M(\alpha, a,b,B,\omega_{ D},\omega_{\Omega})$ so that $$\frac{1}{M}|z-w|^{1/\alpha}\le |f(z)-f(w)|\le M |z-w|^\alpha$$ for every $z,w\in D$.
\end{corollary}

\begin{proof}[Proof of Corollary~\ref{rrje1}]
Let $a$ be a univalent conformal mapping of the unit disk $\mathbf{D}$ onto $D$ and $b$ be a univalent conformal mapping of the unit disk onto $\Omega$. Then in view of Theorem~\ref{teoml}, $b$ and $a^{-1}$ are  $\sqrt{\alpha}-$H\"older continuous.  Then $f=b\circ a^{-1}$, is $\alpha-$H\"older continuous.
\end{proof}

Now we prove the following theorem which deals with H\"older continuity of quasiconformal mappings between smooth domains.

\begin{theorem}\label{lemba}
 Assume that $D$ and $\Omega$ are two Jordan domains  with $\mathscr{C}^1$ boundaries and assume that $a\in D$ and $b\in\Omega$.  Assume further that $\partial D$ and $\partial \Omega$ satisfy $B-$arc-chord condition for some $B\ge 1$. Let $K\ge 1$. Then for every $\beta<1/K$, there is a constant $M_\beta=M(\beta,a,b,\omega_D,\omega_\Omega,B, K)$ so that if $f:D\to \Omega$ is $K-$quasiconformal with $f(a) =b$ then  \begin{equation}|f(z) - f(w)|\le M_\beta|z-w|^\beta, \ \ \ z,w\in D.\end{equation}
\end{theorem}
In connection to Theorem~\ref{lemba}, we want to mention that some more general results are known under some more general conditions on the domains but they do not cover this result.
For example
O. Martio and R. N\"akki in \cite{mar1} showed that if $f$
induces a boundary mapping which belongs to ${\rm
Lip}_\alpha(\partial D)$, then $f$ is in ${\rm Lip}_\beta(D)$, where
$$\beta=\min\{\alpha,1/K\};$$ the exponent $\beta$ is
sharp. We also want to refer to the papers \cite{kot} and \cite{naki} which also consider the global H\"older continuity of quasiconformal mappings. Concerning the integrability of the derivative of a quasiconfromal mapping and its connection to the global H\"older continuity  we refer to the paper by Astala and  Koskela  \cite{astkos}.
\begin{proof}[Proof or Theorem~\ref{lemba}]
Let $\phi:\mathbf{D}\to D$  and $\psi:\Omega \to \mathbf{D}$ be conformal diffeomorphisms so that $\phi(0)=a$ and $\psi(b)=0$. Then $f_0=\psi \circ f\circ \phi$ is a $K-$quasiconfonformal mapping of the unit disk onto itself so that $f_0(0)=0$. Thus by Mori's theorem $$|f_0(z)-f_0(w)|\le 16 |z-w|^{1/K}.$$ Now, if $\beta<1/K$, then there are two constants $\alpha_1<1$ and $\alpha_2<1$ so that $\alpha_1\cdot \alpha_2/K=\beta$. Since $f=\psi^{-1}\circ f_0 \circ \phi^{-1}$, by making use of Corollary~\ref{rrje1}, we get  and $\psi^{-1}$ is $\alpha_1$-H\"older continuous and $\phi^{-1}$ is $\alpha_2$-H\"older continuous. By having in mind the fact that  $f_0$ is $1/K$-H\"older continuous, it follows that $f$ is $\beta-$H\"older continuous as claimed.
\end{proof}
\begin{remark} Similar result can be shown for multiply connected domains in the complex plane having a $\mathscr{C}^1$ boundary.
If $f$ a conformal mapping of the unit disk onto a Jordan domain with merely $\mathscr{C}^1$ boundary, then $f$ is not necessarily Lipschitz continuous. See an example given by Lesley and Warschawski in \cite{les} as well as the example $f_0(z) = 2z+ (1-z) \log (1-z)$ given in the Pommerenke book \cite{boundary}, which is a conformal diffeomorphism of the unit disk onto a Jordan domain with merely $\mathscr{C}^1$ boundary. Then $|f_0'(z)|$ is not bounded and thus $f_0$ is not Lipschitz continuous.   The content of  Corollary~\ref{rrje1} is not new (see for example \cite{Les1}). See also  Warschawski \cite[Corollary,~p.~255]{wars1} for a related result. We should also mention  the paper by Brennan, \cite{brennan}  where the famous Brannen conjecture comes from. Theorem~3 of that paper contains a short proof of special case of \eqref{hardy} for $\Omega=\mathbf{D}$ and $f$ being conformal.
\end{remark}

\section{Auxiliary results}
The starting point of this section is the theorem of Warschawski for conformal mappings which states the following. Assume that $f$ is a conformal mapping of the unit disk onto a Jordan domain $\Omega$ with a $\mathscr{C}^1$ boundary $\gamma$. Assume that $g$ is the arc-length parametrisation of $\gamma$, and assume that $\omega=\omega_{g'}$ is modulus of continuity of $g'$. Assume also that $\gamma$ satisfies  $B-$chord-arc condition for some constant $B>1$. Then for every $p\in\mathbf{R}$, there is a constant $A_p$, depending only on $\Omega$, $\omega$, $B$, $p$ and $f(0)$ so that
\begin{equation}\label{hardye}\int_{\mathbf{T}}|f'(z)|^{p}|\mathrm{d}z|\le E^p_{ p}.\end{equation}

We first give an extension of \eqref{hardy}, and prove a variation of the main result needed in the sequel.
\begin{theorem}\label{mundet}
If $f=g+\overline{h}$ is a $K$- q.c. harmonic mapping of the unit disk $\mathbf{D}$ onto a domain $\Omega$ with $\mathscr{C}^1$ boundary, so that $h$ has holomorphic extension beyond the boundary of the unit disk, then $g', 1/g' \in H^p(\mathbf{D})$ for every $p>0$. Moreover

\begin{equation}\label{hardyeF}\int_{\mathbf{T}}|g'(z)|^{p}|\mathrm{d}z|\le F^p_{ p},\end{equation} where $F_p$ is a constant that depends on the same parameters as $E_p$ in \eqref{hardye} as well as on $k$.
\end{theorem}

Now recall the Morrey inequality.

\begin{proposition}[Morrey's inequality]
Assume that $2<p\le \infty$ and assume that $U$ is a bounded domain in $\mathbf{R}^2$ with $\mathscr{C}^1$ boundary. Then there exists a constant $C$ depending only on  $p$ and $U$ so that \begin{equation}\label{morrey}
\|u\|_{\mathscr{C}^{0,\alpha}(U)}\le C\|u\|_{W^{1,p}(U)}
\end{equation}
for every $u\in \mathscr{C}^1(U)\cap L^p(U)$, where
$$\|u\|_{\mathscr{C}^{0,\alpha}(U)}=\sup_{z\neq w} \frac{|u(z)-u(w)|}{|z-w|^\alpha},$$ and
$$\alpha=1-\frac{2}{p},$$ and $$\|u\|_{W^{1,p}(U)}=\|u\|_{L^{p}(U)}+\|Du\|_{L^p(U)}.$$

Here $W^{1,p}(U)$ is the Sobolev space.

\end{proposition}
\begin{corollary}\label{rrje}
Under the conditions of the previous theorem, for every $\alpha<1$, $f$ and $f^{-1}$ are $\alpha-$H\"older continuous. The result is optimal since, $f$ is not necessarily Lipschitz in general.
\end{corollary}

\begin{remark}
If $h\equiv 0$, then Theorem~\ref{mundet} reduces to the classical result of Warschawski (\cite{wars}), see also a similar result by Smirnov \cite{smirnov} and Goluzin \cite[Theorem~7, p.~415]{G}. We include the proof of Theorem~\ref{mundet} for the completeness of the argument.
\end{remark}

\begin{proof}[Proof of corollary~\ref{rrje}]
Let $\alpha<1$ and prove that $f$ is $\alpha-$H\"older continuous. We have $$|f(e^{it})-f(e^{is})|=\int_{s}^t |\partial_\tau f(e^{i\tau})|\mathrm{d}\tau\le
\left(\int_{s}^t |\partial_\tau f(e^{i\tau})|^p\mathrm{d}\tau\right)^{1/p}\left(\int_{s}^t \mathrm{d}\tau\right)^{1/q}.$$ Therefore for $\alpha=1-1/p=1/q$ we get

$$|f(e^{it})-f(e^{is})|\le \|\partial_\tau f\|_p |s-t|^\alpha.$$

As $h$ is smooth in $\overline{\mathbf{D}}$, it follows that $g$ is $\alpha-$H\"older continuous in $\mathbf{T}$. By using the well-known Hardy-Littlewood   theorem \cite[Theorem~4,~p.413]{G}, we get that $g$ is $\alpha-$H\"older continuous  on $\mathbf{D}$. Thus $f$ is $\alpha-$H\"older continuous  on $\mathbf{D}$.

To prove that $f^{-1}$ is $\alpha-$H\"older continuous, observe that for $w=f(z)$,  $$\partial_w f^{-1}(w)=\frac{\overline{f_z}}{J_f}= \frac{\overline{g'(z)}}{|g'(z)|^2-|h'(z)|^2}.$$ Thus \[\begin{split}\int_{\Omega}|\partial_w f^{-1}(w)|^p \mathrm{d}\lambda(w)&=\int_{\mathbf{D}}\left( \frac{|g'(z)|}{|g'(z)|^2-|h'(z)|^2}\right)^p J_f \mathrm{d}\lambda(z)\\&\le \int_{\mathbf{D}} \frac{|g'(z)|^{p+2}}{|g'(z)|^{2p}}\frac{1+k^2}{(1-k^2)^p}\mathrm{d}\lambda(z)\\&=
\frac{1+k^2}{(1-k^2)^p}\int_{\mathbf{D}}|g'(z)|^{2-p} \mathrm{d}\lambda(z), \ \ k=(K-1)/(K+1).\end{split}\]
Here $\lambda$ is the Lebesgue measure in the plane.
Therefore by using the isoperimetric inequality for holomorphic functions we get
\[\begin{split}\int_{\Omega}|D f^{-1}(w)|^p \mathrm{d}\lambda(w)&\le  \frac{(1+k^2)(1+k^p)}{(1-k^2)^p}\int_{\mathbf{D}}|g'(z)|^{2-p} \mathrm{d}\lambda(z)
\\&\le \frac{(1+k^2)(1+k^p)}{4\pi (1-k^2)^p}\left(\int_{\mathbf{T}}|g'(z)|^{1-p/2} |\mathrm{d}z|\right)^2<\infty.\end{split}\]

From \eqref{morrey} we infer that $u=f^{-1}$ is $\alpha-$H\"older continuous and the corollary is proved.

\end{proof}

\begin{proof}[Proof of Theorem~\ref{mundet}]

We use the following proposition
\begin{proposition}\label{lindel}\cite{aim2}
If $f(z)=\mathcal{P}[f^*](z)$ is a quasiconformal harmonic mapping of the unit disk onto a Jordan domain bounded by a curve  $\gamma$, then the function $$U(z):=\arg\left(\frac{1}{z}\frac{\partial}{\partial \varphi}f(z)\right)$$ is a well defined and smooth in $\mathbf{D}^*:=\mathbf{D}\setminus\{0\}$ and has a continuous extension to $\mathbf{T}$ if and only if $\gamma\in \mathscr{C}^1$. Furthermore, there holds $$U(e^{i\varphi})=\beta(\varphi)-\varphi,$$  where  $\beta(\varphi)$ is the tangent angle of $\gamma$ at $f^*(e^{i\varphi})$.
\end{proposition}



By the assumption we have that $h(z) =\sum_{j=0}^\infty b_j z^j$ for $|z|<\rho$, where $\rho$ is a certain constant bigger than $1$.

Therefore, the mapping $$h_1(z) =\frac{1}{z}\overline{h'\left(\frac{1}{\bar z}\right)}= \sum_{j=0}^\infty \frac{j \overline{b_j}}{z^j}$$ is well defined holomorphic function in the domain $D_1=\{z:|z|>1/\rho\}$.

Since $\Gamma=\partial\Omega$ is rectifiable, for $z=re^{it}$, we have that $$F(z)= \partial_t f(re^{it})= i z g'(z) -i \overline{zh'(z)} \in h^1(\mathbf{D}),$$ (see e.g. \cite{kmm,MP}). Therefore, by having in mind the quasiconformality, we get that $g', h'\in H^1(\mathbf{D})$.   In particular, there exist non-tangential limits of those functions almost everywhere on $\mathbf{T}$. We recall that $ h^1(\mathbf{D})$ and  $H^1(\mathbf{D})$ are the Hardy classes of harmonic and holomorphic functions, respectively, defined in the unit disk $\mathbf{D}$.

 Let $$H(z) = i\left(z g'(z) - \frac{1}{z} h_1(z)\right),\ \ \  1/\rho<|z|<1.$$ Then, for almost every $t\in[-\pi,\pi]$, we have $$\lim_{r\to 1}H(re^{it})= \lim_{r\to 1} F(re^{it}).$$
Then there is a set of points $0<\varphi_1<\varphi_2<\varphi_3<\varphi_4<2\pi$ so that

\begin{equation}\label{fi12}\lim_{r\to 1}H(r e^{i\varphi_j})=H(e^{i\varphi_j}), \end{equation} exist for every $j=1,2,3,4$.

Let $1<R<\rho$ and let $S_1=\{z=r e^{i\phi}; \phi\in(\varphi_1,\varphi_4), r\in(1/R,1)\}$,  $S_2=\{z=r e^{i\phi}; \phi\in(\varphi_3,2\pi+\varphi_2), r\in(1/R,1)\}$ and let  $w=\Phi_j(z)$ be a conformal mapping of the unit disk onto the region  $S_j$ so that
 \begin{equation}\label{phi00}\Phi_1(0)= \frac{1}{2}\left(\frac{1}{R}+1\right) e^{i/2 (\varphi_1+\varphi_4)}, \ \  \Phi_2(0)= -\frac{1}{2}\left(\frac{1}{R}+1\right) e^{i/2 (\varphi_2+\varphi_3)}.\end{equation}
 Let $s_1, s_2, s_3, s_4\in[0,2\pi]$ so that $\varphi_1<s_1<s_2<\varphi_2$, and $\varphi_3<s_3<s_4<\varphi_4$. Then $$\{e^{is}:s\in(s_1,s_4)\cup (s_3,2\pi +s_2)\}=\mathbf{T}. $$

 Observe that $\mathbf{T}\subset D_1$.

 Define the holomorphic mapping
$K_j(z) = H(\Phi_j(z))$, $z\in \mathbf{D}$, $j=1,2$.
In view of \eqref{fi12}, we have that $H$ is bounded on the boundary arcs  $I_j=[1/R,1] e^{i\varphi_j}$, $j=1,4$ of $S$. Also it is clear that it is bounded in the inner arc. Therefore $K_j$ is a non-vanishing bounded analytic function defined in the unit disk.
Let $L_j(z) = \log K_j(z)$. Then for $j=1,2$ $$v_j(z)=\Im L_j(z) = \mathrm{arg}(K_j(z)),$$ is a bounded harmonic function, so that $\lim_{r\to 1} v_j(re^{it}) = v_j(e^{it})$ is a continuous function on the unit circle.

To show that $v$ is a bounded well-defined function, observe that $$H(z) = zg'\left(1-\frac{h_1(z)}{z^2}\right),$$ and so $$\mathrm{arg}\, H(z) =\mathrm{arg}\, (z g') +\mathrm{arg}\, \left(1-\frac{h_1(z)}{z^2 g'(z)}\right).$$ First of all  for $|z|$ close to $1$, the function $$\Re \left(1-\frac{h_1(z)}{z^2 g'(z)}\right)$$ is bigger than $1-(1+k)/2$, where $k$ is the constant of quasiconformality. On the other hand, in view of Proposition~\ref{lindel},  $i(g'- \overline{z h'}/z)= f_t(e^{it})/z$ has a continuous argument in the punctured disk $0<|z|\le1$. Since $\Re(1-\overline{z h'}/(z g'))>0$, we obtain that $\mathrm{arg}(g')$ is well-defined and bounded function close to the boundary of the unit disk.

We can also choose $R$ close enough to $1$ so that the variation of the argument:
\begin{equation}\label{closeen}
 \Delta_{\mathbf{T}} \mathrm{arg}K_j(e^{it})\le 1+ \Delta_{\mathbf{T}} \mathrm{arg} H_j(e^{is}).
\end{equation}

Assume that $\epsilon>0$ so that $\epsilon |p| <\pi/2$ and let \begin{equation}\label{polinom}P_j(t)=a_{j,0}+\sum_{m=1}^n c_m \cos mt + d_m \sin m t\end{equation} be a trigonometric polynomial so that \begin{equation}\label{vvv}|v_j(e^{it})-P_j(t)|\le \epsilon\end{equation} for $t\in[0,2\pi]$. Let $\Psi$ be the holomorphic function, so that $\Im (\Psi(e^{it})) = P_j(t)$ and $\Psi_j(0)=a_{j,0}$.

Observe that $$a_{j,0} = \frac{1}{2\pi}\int_{-\pi}^\pi P_j(t) dt$$ and
\begin{equation}\label{aj0}|a_{j,0}|\le \frac{1}{2\pi}\int_{-\pi}^\pi |P_j(t)| dt\le \epsilon + \frac{1}{2\pi}\int_{-\pi}^\pi |v_j(t)| dt.\end{equation}
Then for every $r\in(0,1)$ we have

 $$\int_{0}^{2\pi} e^{p (L_j(r e^{it})-\Psi(r e^{it}) )}\frac{\mathrm{d}t}{2\pi }=  e^{p (\Psi(0)- L_j(0))}.$$
 So by taking the real part and letting $r\to 1$ we get $$\int_{0}^{2\pi} e^{p \Re (L_j( e^{it})-\Psi(e^{it}) )}\cos \left\{p\Im \left[ L_j( e^{it})-\Psi(e^{it})\right]\right\}\frac{\mathrm{d}t}{2\pi }=\Re  e^{p (L_j(0)-\Psi(0))}.$$
Thus $$\int_{0}^{2\pi} e^{p \Re (L_j( e^{it})-\Psi(e^{it}) )}\frac{\mathrm{d}t}{2\pi }\le   \frac{\abs{\Re  e^{p (L_j(0)-\Psi(0))}}}{\cos p \epsilon}.$$ Therefore

$$\int_{0}^{2\pi} e^{p \Re (L_j( e^{it}))}\frac{\mathrm{d}t}{2\pi }\le  \max_{t\in[0,2\pi]}e^{p \Re (\Psi(e^{it}))} \frac{\abs{\Re  e^{p (L_j(0)-\Psi(0))}}}{\cos p\epsilon}=G_p.$$

The constant $G_p$ depends on the same parameters as the constant $E_p$ from \eqref{hardye} together with the constant of quasiconformality $k$, and this follows  from the fact that $\Psi(0)=a_{j,0}$, \eqref{phi00}, \eqref{aj0}, \eqref{closeen} and a Cauchy type inequality for $H(z)$ in the annulus $1/R<|z|<1$, where $1/R=(1/\rho+1)/2$. Here $\rho$ is a given constant bigger than $1$ as in the begging of the proof.

Since $p \Re L_j(z) = p \log |K_j(z)|$, it follows that $\exp (p\log |K_j(z)|)=|K_j(z)|^p$. Therefore $K_j\in H^p$. Now we have
\[\begin{split}\int_{\mathbf{T}}|H(e^{is})|^p \mathrm{d}s &\le \int_{\{e^{is}:s_1\le s\le s_4\}}|H(e^{is})|^p \mathrm{d}s+\int_{\{e^{is}:s_3\le s\le s_2+2\pi\}}|H(e^{is})|^p \mathrm{d}s\\&=\int_{\{e^{it}:t_1\le t\le t_4\}}|H(\Phi_1(e^{it}))|^p |\Phi_1'(e^{it})|\mathrm{d}t\\&+\int_{\{e^{it}:t_3\le t\le t_2+2\pi\}}|H(\Phi_2(e^{it}))|^p |\Phi_2'(e^{it})|\mathrm{d}t,\end{split}\]
where $\Phi_1(t_i)=s_i$, $i=1,4$, and $\Phi_2(t_i)=s_i$, $i=2,3$. Moreover $|\Phi_1'(e^{it})|$ is bounded on $\{e^{it}:t_1\le t\le t_4\}$ and $|\Phi_2'(e^{it})|$ on $\{e^{it}:t_3\le t\le t_2+2\pi\}$. Therefore \[\begin{split}\int_{\mathbf{T}}|H(e^{is})|^p \mathrm{d}s&\le C \int_{\{e^{it}:t_1\le t\le t_4\}}|H(\Phi_1(e^{it})|^p \mathrm{d}t\\&+C \int_{\{e^{it}:t_3\le t\le t_2+2\pi\}}|H(\Phi_2(e^{it})|^p \mathrm{d}t\le C(\|K_1\|_p^p+\|K_2\|_p^p)\\ &\le L_p^p<\infty.\end{split}\]

The constant $L_p$ depends on the same parameters as  $E_p$ from \eqref{hardye} and the quasiconformal constant $k$.

Thus $H\in H^p(\mathbf{D})$, and so $f_t\in h^p(\mathbf{D})$. Since $f$ is quasi-conformal, it follows that $g'\in H^p$.

\end{proof}

\begin{lemma}\label{lomba}
Let $\alpha\in(0,1)$. Then there is a positive constant $C(\alpha)>1$ satisfying the following property. If $f$ is a holomorphic function defined in the unit disk with continuous extension up to the boundary and if  $$X=\sup_{e^{it}\neq e^{is}}\frac{|f(e^{it}) - f(e^{is})|}{ |e^{it}-e^{is}|^\alpha}$$ and   $$Y= \sup_{|z|<1}(1-|z|)^{1-\alpha}|f'(z)|,$$ then \begin{equation}\label{ccal}\frac{1}{C(\alpha)}X\le Y \le C(\alpha) X.\end{equation}
\end{lemma}
\begin{remark}
We want to mention that a result similar to Lemma~\ref{lomba} is probably valid for the more general classes of mappings such as, real harmonic functions, or quasiconformal harmonic mappings, but we do not need such results (see e.g. \cite{MP1}).
\end{remark}
\begin{proof}
First we have for $z=re^{i\theta}$ that

$$f'(z) =\frac{1}{2\pi}\int_{-\pi}^\pi \frac{f(e^{it})e^{it} \mathrm{d}t}{(e^{it}-z)^2}=\frac{1}{2\pi} \int_{-\pi}^\pi \frac{(f(e^{it})-f(e^{i\theta}))e^{it} \mathrm{d}t}{(e^{it}-z)^2}.$$ Therefore

$$|f'(z)|\le \frac{1}{2\pi}\int_{-\pi}^\pi \frac{|f(e^{i(t+\theta)})-f(e^{i\theta})|}{1+r^2-2r \cos t} \mathrm{d}t .$$
By using the inequality $|e^{it}-1|\le |t|$, and introducing the change of variables $\varphi=2t\sqrt{r}/\pi $,  it follows that
\[
\begin{split}|f'(z)|&\le \frac{X}{\pi} \int_0^\pi \frac{t^\alpha}{(1-r)^2 +\frac{4r}{\pi^2}t^2} \mathrm{d}t
\\&=\frac{X\pi^\alpha}{2^{\alpha+1}}\frac{r^{-\frac{1+\alpha}{2}}}{(1-r)^{1-\alpha}}\int_0^{ \frac{2\sqrt{r}}{1-r}} \frac{\varphi^\alpha \mathrm{d}\varphi}{1+\varphi^2}
\\&\le\frac{X\pi^\alpha}{2^{\alpha+1}}\frac{r^{-\frac{1+\alpha}{2}}}{(1-r)^{1-\alpha}}\int_0^\infty \frac{\varphi^\alpha \mathrm{d}\varphi}{1+\varphi^2}.\end{split}\]
So for $r>1/2$ we have
$$(1-|z|)^{1-\alpha}|f'(z)|\le X \frac{\pi^\alpha}{2^{\alpha+1}} 2^{\frac{1+\alpha}{2}}\int_0^\infty \frac{\varphi^\alpha}{1+\varphi^2} \mathrm{d}s.$$
Thus, after length but elementary calculation we get that   $$(1-|z|)^{1-\alpha}|f'(z)|\le X   \frac{\pi ^{1+\alpha }}{2^{\frac{1+3 \alpha }{2}}} \sec\left[\frac{\pi  \alpha }{2}\right].$$
For $r<1/2$ we have  \[\begin{split}(1-|z|)^{1-\alpha} |f'(z)|&\le \frac{X (1-r)^\alpha}{\pi} \int_0^\pi \frac{t^\alpha}{(1-r)^2 +\frac{4r}{\pi^2}t^2} \mathrm{d}t \\&\le \frac{X (1-r)^\alpha}{\pi} \int_0^\pi \frac{t^\alpha}{(1-r)^2 } \mathrm{d}s\\&\le X 2^{2-\alpha}\frac{\pi^{\alpha}}{\alpha+1}. \end{split}\]
Conversely, by using the proof of Hardy-Littlewood theorem (\cite[Theorem~3,~p.~411]{G}) if $$(1-|z|)^{1-\alpha}|f'(z)|\le Y,$$ then for $|s-t|\le 1$ we get
$$|f(e^{it}) - f(e^{is})|\le  Y(2/\alpha+1) |t-s|^\alpha.$$
Therefore for $t,s\in [-\pi,\pi]$, by noticing that $e^{it}=e^{it +2\pi i }$, for the case $|t-s|>1$ or for the case $|2\pi - (t-s)|>1$ we get
\[\begin{split}|f(e^{it}) - f(e^{is})|&\le \sum_{j=1}^4 |f(e^{it_j}) - f(e^{it_{j-1}})|\\&\le  \sum_{j=1}^4 Y(2/\alpha+1) |t_j-t_{j-1}|^\alpha\le 4 Y(2/\alpha+1)|t-s|^\alpha.\end{split}\]
So \eqref{ccal} is satisfied for
$$C(\alpha) =\max\left\{2^{2-\alpha}\frac{\pi^{\alpha+1}}{\alpha+1}, \frac{\pi ^{1+\alpha }}{2^{\frac{1+3 \alpha }{2}}} \sec\left[\frac{\pi  \alpha }{2}\right], 4\left(\frac{2}{\alpha}+1\right)\right\}.$$
\end{proof}

\section{Proof of main result (Theorem~\ref{teoml})}
We divide the proof into two cases.
\begin{enumerate}
\item[a)] $D$ is the unit disk $\mathbf{D}$,

\item[ b)] $D$ is a general Jordan domain with a $\mathscr{C}^1$ boundary.
\end{enumerate}
a) Since $\gamma\in \mathscr{C}^1$, $\gamma$ has the following property. For every point $p\in\gamma$ there are complex numbers $|a|=1$ and $b$ so that the parametrisation of the curve \begin{equation}\label{gag}\gamma_p =a\cdot(\gamma-p)\end{equation} above the point $0$ has the form $\eta_p(x) =(x, \varphi_p(x)),$  so that $\varphi_p(0)=\varphi_p'(0)=0$.

Further for every $p$ and every  $\epsilon>0$, there is $\delta_0=\delta_0(\epsilon)$ so that   $$|\varphi_p(x) - \varphi_p(0)-\varphi_p'(0) x|\le \epsilon |x|, $$ for $|x|\le \delta$. Moreover, $\delta_0$ can be chosen to be independent on $p$. I.e. it depends on $\epsilon$ and $\gamma$ only.

Let $x(t) =\Re (f(e^{it}))$. Then locally $y(t)=\Im (f(e^{it}))=\varphi(x(t))$. Assume also that $x(0) = 0$ and $f(1)=(0,0)$. For fixed $\epsilon>0$, because of Theorem~\ref{lemba} there is $\delta>0$ ($\delta<1$) so that $|t|\le \delta$ implies $|x(t)|\le \delta_0$ and so that

\begin{equation}\label{cruci}|\varphi_p(x(t)) - \varphi_p(0)- \varphi_p'(0)x(t)|\le \epsilon |x(t)|.\end{equation}
Since $\varphi_p(0)=\varphi_p'(0)=0$ we get
\begin{equation}\label{cruci1}|\varphi_p(x(t)) |\le \epsilon |x(t)|, \ \ \  |t|\le \delta.\end{equation}
Let \begin{equation}\label{vv}v(z) = \Im f(z)=\Im (g+\bar h)=\Re (i (h(z)-g(z)))\end{equation} and \begin{equation}\label{uu}u(z) = \Re f(z) = \Re (g(z)+h(z)) .\end{equation}
Then by the Schwarz formula we get $$i (h(z)-g(z))= i \Im (h(0)-g(0))+\frac{1}{2\pi}\int_{-\pi}^\pi \frac{e^{is}+z}{e^{is}-z} \tilde v(s) \mathrm{d}s$$ where \begin{equation}\label{tildev}\tilde v(s) = \Re (i (h(e^{is})-g(e^{is}))).\end{equation}
Thus

\begin{equation}\label{divide}i (h'(z)-g'(z)) =\frac{1}{\pi}\int_{-\pi}^\pi \frac{\tilde v(s) - \tilde v(0)}{(e^{is}-z)^2} \mathrm{d}s.\end{equation}

From now on we divide the proof into two steps.

\subsection{Assume first  that $f$ is $\alpha^{1/2}-$H\"older continuous  and prove that the H\"older constant do not depend on $f$}
Since $f=g+\bar h$ is $\alpha^{1/2}-$H\"older continuous, then $(1-|z|)^{1-\alpha^{1/2}} (|h'|+|g'|)$ is bounded, and so  the following maximum $$A=\max_{|z|<1}(1-|z|)^{1-\alpha} |i(h'(z)-g'(z))|$$ is attained in a point of the unit disk. We can assume that $A= (1-\rho)^{1-\alpha} |i(h'(\rho)-g'(\rho))|$ for some $\rho\in [0,1)$.
Then we get $$B=\max_{|z|<1}(1-|z|)^{1-\alpha} (|h'(z)|+|g'(z)|)\le K A,$$ where $K$ is the constant of the quasiconformality. In particular, from Lemma~\ref{lomba}, $h$ and $g$ are $\alpha-$H\"older's continuous on the boundary $\mathbf{T}$. More precisely

$$|h(e^{it})-h(e^{is})|\le KA C(\alpha)|e^{it}-e^{is}|^\alpha$$ and

$$|g(e^{it})-g(e^{is})|\le KA C(\alpha)|e^{it}-e^{is}|^\alpha.$$

Therefore $$|f(e^{it})-f(e^{is})|\le 2KA C(\alpha)|e^{it}-e^{is}|^\alpha.$$
In particular for $\tilde u (s) = \Re (f(e^{is}))=\Re (g(e^{it})+h(e^{it})))$ we have
\begin{equation}\label{inpart}|\tilde u(s)-\tilde u(0)|\le 2K A C(\alpha)|s|^\alpha.\end{equation}

Then, having in mind that for $t\in (-\delta,\delta)$, $\tilde v(t) = \varphi(\tilde u(t))$, from \eqref{divide}, \eqref{cruci} and the proof of Lemma~\ref{lomba},  we get

\[\begin{split}|i (h'(\rho)-g'(\rho))|(1-\rho)^{1-\alpha} &\le (1-\rho)^{1-\alpha}\int_{-\pi}^\pi \frac{|\tilde v(s) - \tilde v(0)|}{\rho^2-2\rho \cos s +1} \frac{\mathrm{d}s}{\pi}\\&=(1-\rho)^{1-\alpha}\int_{[-\delta,\delta]} \frac{|\tilde v(s) - \tilde v(0)|}{\rho^2-2\rho \cos s +1} \frac{\mathrm{d}s}{\pi}\\&+(1-\rho)^{1-\alpha}\int_{[-\pi,\pi]\setminus[-\delta,\delta]} \frac{|\tilde v(s) - \tilde v(0)|}{\rho^2-2\rho \cos s +1} \frac{\mathrm{d}s}{\pi}
\\&\le2\epsilon K A C(\alpha)\int_{[-\delta,\delta]} \frac{(1-\rho)^{1-\alpha}|s|^\alpha}{\rho^2-2\rho \cos s +1} \frac{\mathrm{d}s}{\pi}+Z
\\&\le 2 \epsilon K A C^2(\alpha)+Z,
\end{split}
\]
 where

$$Z=(1-\rho)^{1-\alpha}\int_{[-\pi,\pi]\setminus[-\delta,\delta]} \frac{|\tilde v(s) - \tilde v(0)|}{\rho^2-2\rho \cos s +1} \frac{\mathrm{d}s}{\pi}.$$
Further $$Z\le \mathrm{diam}(\Omega)\frac{2\pi }{\pi}\frac{1}{1+\cos^2\delta - 2\cos \delta \cdot \cos \delta}=\frac{2\mathrm{diam}(\Omega)}{\sin^2\delta}.$$
So $$A\le 2 \epsilon K A C^2(\alpha)+X\le 2 \epsilon K A C^2(\alpha)+\frac{2\mathrm{diam}(\Omega)}{\sin^2\delta}.$$
By choosing  $\epsilon>0$ so that  $$2 \epsilon K A C^2(\alpha)<A/2,$$ we get
\begin{equation}\label{adelta}A\le  \frac{4\mathrm{diam}(\Omega)}{\sin^2\delta}.\end{equation}

Observe that $\delta$, and so $A$ depends on $K$, $\gamma$, $\alpha$ and modulus of continuity of $f$ at the boundary, but not on a specific point $z\in \mathbf{D}$.

\subsection{Let us remove the assumptions $f$ is $\sqrt{\alpha}-$ H\"older continuous and use Approximation argument}

If $p\in \partial \Omega= \gamma$ and $\gamma\in \mathscr{C}^1$, then, after possible rotation and translation of $\Omega$ (similarly as in \eqref{gag}), which preserves the harmonicity and the quasiconformal constant of the corresponding mapping, we can assume that $p=0$ and the unit normal vector is $N_p=(1,0)$. So we can find a sub-arc of $\gamma$ containing $p$ at its interior which  is the graphic of a function defined as follows $$\gamma_p(\eta)=\{(x,\phi(x)): x\in (-\eta,\eta)\}.$$ We also can assume that $\eta>0$ is a positive constant that depends only on $\gamma$ but not on the specific point $p$. Then we have $\phi'(0)=0$. Let $\Omega_p\subset \Omega$ be a Jordan domain bounded by a $\mathscr{C}^1$ Jordan curve $\Gamma_p$ consisted of $\gamma_p(\eta/2)$ and an interior part, which we denote by  $\chi_p(\eta)$, which is subset of $\Omega$ and assume that $a_p\in\Omega_p$ be a fixed point. Then for small enough $\sigma=\sigma(\gamma)>0$, the domain $\Omega_p(\kappa)=\Omega_p - \kappa N_p$ is a subset  of $ \Omega$, for every $\kappa\in[0,\sigma]$.

Let $\Phi_{p,\kappa}: \mathbf{D}\to f^{-1}(\Omega_p(\kappa))$ be a conformal mapping so that $$\Phi_{p,\kappa}(0) =f^{-1}(a_p-\kappa N_p).$$

Since $\mathbf{T}$ is compact, there is a finite family of Jordan domains $\Omega_{p_j},\ \ j=1,\dots,n$ so that $T_j:=f^{-1}(\partial \Omega\cap \partial \Omega_{p_j})$, $j=1,\dots,n$ covers $\mathbf{T}$. Moreover, $f\circ \Phi_{p_j,\kappa}: \mathbf{D}\to \Omega_{p_j}$ is $\sqrt{\alpha}-$H\"older continuous in $\mathbf{D}$, because $f$ is smooth in $\Phi_{p_j,\kappa}( \mathbf{D})$ and $\Phi_{p_j,\kappa}$ is ${\alpha}^{1/2}-$H\"older continuous because of Corollary~\ref{rrje}. Further, in view of the first case, there is a constant $A_{p_j}$ (see Lemma~\ref{lomba} and \eqref{adelta}) which depends only on $\Omega_{p_j}$ and $\alpha$ so that
  $$|f\circ \Phi_{p_j,\kappa}(e^{it})-f\circ \Phi_{p_j,\kappa}(e^{is})|\le A_{p_j} |e^{is}-e^{it}|^{\alpha^{2/3}}.$$  Note that $A_{p_j}$ also depends on the modulus of continuity of $f\circ \Phi_{p_j,\kappa}$ where $ \kappa\in[0,\sigma]$, but this family is uniformly continuous, and we can choose modulus of continuity that does not depend on $\kappa$, so $A_{p_j}$ will not depend on $\kappa$ either. Namely the $K-$quasiconformal mappings $G_\kappa:=\kappa+ f\circ \Phi_{p_j,\kappa}$, $\kappa\in[0,\sigma]$, map the unit disk onto $\Omega_{p_j}\in \mathscr{C}^1$ and satisfy the condition $G_\kappa(0)=a_{p_j}$.  By letting $\kappa\to 0$ we get

   $$|f\circ \Phi_{p_j,0}(e^{it})-f\circ \Phi_{p_j,0}(e^{is})|\le A_{p_j} |e^{is}-e^{it}|^{{\alpha}^{2/3}}.$$

  Therefore, by having in mind the fact that $\Phi^{-1}_{p_j,0}$ is $\alpha^{1/3}-$H\"older continuous on $T_j$ (in view of Corollary~\ref{rrje}), we conclude that $f$ is $\alpha-$H\"older continuous in $T_j'\subset T_j$, where $T_j'$ is a little bit smaller arc, but so that $\mathbf{T}\subset \bigcup_{j=1}^n T_j'$. Thus, $f$ is $\alpha-$H\"older continuous in $\mathbf{T}$. By the standard argument we now obtain that $f$ is $\alpha-$H\"older continuous in $\mathbf{D}$, concluding the case a).

  Notice that $\alpha>0$ is an arbitrary number smaller than $1$, so $f$ is also $\alpha^{1/2}-$H\"older continuous.

Hence, if we want to get more explicit estimate of $A$, then we repeat one more time the procedure proceed  in the previous subsection, but with $$A=\sup_{|z|<1}(1-|z|)^{1-\alpha}|i(g'(z)-h'(z))|,$$ and thus we get the estimate \begin{equation}\label{adelta1}A-\varepsilon\le  \frac{4\mathrm{diam}(\Omega)}{\sin^2\delta},\end{equation} instead of \eqref{adelta} for arbitrary $\varepsilon>0$, and thus \eqref{adelta} is valid also in this case. Further,
\[\begin{split}|Df(z)|=(|g'(z)|+|h'(z)|)&\le K (|g'(z)-h'(z)|)\\&\le K A (1-|z|)^{1-\alpha},\end{split}\] and so that
\begin{equation}\label{diskineq}\begin{split}\int_{\mathbf{D}}|Df(z)|^p \mathrm{d}\lambda(z) &\le K^p\int_{\mathbf{D}}A^p(1-|z|)^{(1-\alpha) p}\mathrm{d}\lambda(z)\\&= \frac{2  \pi K^p A^p }{2-3 (1-\alpha) p+(1-\alpha)^2 p^2}=C^p_{p,\alpha},\end{split}\end{equation} for $(1-\alpha) p<1$.  For example, by choosing $\alpha=1-1/(2p)$, we get $$C_p^p=   \frac{8}{3}\pi K^p A^p .$$

b) The H\"older continuity follows from the case a) and Corollary~\ref{rrje}. To deal with the integral, we use the change of variables. Namely, let $\phi:\mathbf{D}\to D$ be a biholomorphism so that $\phi(0)=a$. Then by using H\"older's inequality, isoperimetric inequality and relations  \eqref{hardye} and \eqref{diskineq}  we get  \[\begin{split}\int_{D}|Df(z)|^p \mathrm{d}\lambda(z)&=\int_{\mathbf{D}}(|Df(\phi(\zeta))|\cdot|\phi'(\zeta)|)^p |\phi'(\zeta)|^{2-p} \mathrm{d}\lambda(\zeta)\\&=\int_{\mathbf{D}}(|Df(\phi(\zeta))|\cdot|\phi'(\zeta)|)^p |\phi'(\zeta)|^{2-p} \mathrm{d}\lambda(\zeta)\\&\le \left(\int_{\mathbf{D}}(|Df(\phi(\zeta))|\cdot|\phi'(\zeta)|)^q \mathrm{d}\lambda(\zeta)\right)^{p/q}\\ & \ \ \ \ \times \left(\int_{\mathbf{D}} |\phi'(\zeta)|^{(2-p)q'} \mathrm{d}\lambda(\zeta)\right)^{1/q'}\\ &
\le C_q^{p} \cdot \frac{1}{(4\pi)^{1/q'}} \left(\int_{\mathbf{T}} |\phi'(\zeta)|^{(1-p/2)q'} \mathrm{d}\lambda(\zeta)\right)^{2/q'}\\&\le C_q^{p} \cdot \frac{1}{(4\pi)^{1/q'}} \left(E_{(1-p/2)q'}\right)^{2-p}=B^p_p,\end{split}\] where $1/q+1/q'=1$, and $q=p+1$.

\section*{Acknowledgement}  I am very grateful to the referee   for numerous typographic and stylistic corrections that have improved this paper.


\begin{thebibliography}{1}

\bibitem{Ahl} \textsc{L. Ahlfors:} \textit{Lectures on Quasiconformal mappings.}
Van Nostrand Mathematical Studies, D. Van Nostrand 1966.

\bibitem{astkos}
\textsc{K. Astala, P.  Koskela:}
\emph{Quasiconformal mappings and global integrability of the derivative.}
J. Anal. Math. \textbf{57}, 203--220 (1991).
\bibitem{astala}
\textsc{K. Astala,
V.  Manojlovi\'c}:
\emph{On Pavlovi\'c theorem in space.}
Potential Anal. \textbf{43}, No. 3, 361--370 (2015).

\bibitem{abr}
\textsc{S. Axler, P. Bourdon and W. Ramey:} {\it Harmonic function theory}.
Springer-Verlag, New York 1992.

\bibitem{abu}
\textsc{V. Bo\v zin, M. Mateljevi\'c}: \emph{Quasiconformal and HQC mappings between Lyapunov Jordan domains}.
To appear in Ann. Sc. Norm. Super. Pisa, Cl. Sci. (5)
DOI Number: 10.2422/2036-2145.201708\_013.


\bibitem{brennan}
\textsc{J. E. Brennan:}
\emph{The integrability of the derivative in conformal mapping.}
J. Lond. Math. Soc., II. Ser. \textbf{18}, 261--272 (1978).

\bibitem{chen}
\textsc{Sh. Chen and S. Ponnusamy:}  \emph{John disks and $K$-quasiconformal  harmonic mappings.} J. Geom. Anal. \textbf{17}(2017), 1468--1488.

\bibitem{duren}
\textsc{P. Duren:} {\it Harmonic mappings in the plane.} Cambridge University
Press, 2004.

\bibitem{HE}
\textsc{E. Heinz:} {\it On one-to-one harmonic mappings.} Pac. J. Math. \textbf{9},
101--105 (1959).
\bibitem{hs2}
\textsc{W. Hengartner,
G. Schober:}
\emph{Univalent harmonic functions.}
Trans. Am. Math. Soc. \textbf{299}, 1--31 (1987).

\bibitem{hs3}
\textsc{W. Hengartner,
G. Schober:}
\emph{Harmonic mappings with given dilatation.}
J. Lond. Math. Soc., II. Ser. \textbf{33}, 473--483 (1986).

\bibitem{gm}
\textsc{F. W. Gehring, O.  Martio:}
\emph{Lipschitz classes and quasiconformal mappings.}
Ann. Acad. Sci. Fenn., Ser. A I, Math. \textbf{10}, 203--219 (1985).

\bibitem{G}
\textsc{G. M. Goluzin:} {\it Geometric function theory of a Complex Variable}. --Transl. Of Math. Monographs 26. - Providence: AMS, 1969.
\bibitem{pisa}
\textsc{D. Kalaj:}
\emph{Harmonic mappings and distance function.}
Ann. Sc. Norm. Super. Pisa, Cl. Sci. (5) \textbf{10}, No. 3, 669--681 (2011).
\bibitem{pacific}
\textsc{D. Kalaj:}
\emph{Quasiconformal harmonic mappings between Dini smooth Jordan domains}. Pac. J. Math. \textbf{276}, 213--228 (2015).
\bibitem{plane}
\textsc{D. Kalaj,
M. Pavlovi\'c}:
\emph{Boundary correspondence under quasiconformal harmonic diffeomorphisms of a half-plane.}
Ann. Acad. Sci. Fenn., Math. \textbf{30}, No. 1, 159--165 (2005).
\bibitem{aim2}
\textsc{D. Kalaj}: \emph{Muckenhoupt weights and Lindel\"of theorem for harmonic mappings.}
Adv. Math. \textbf{280}, 301--321 (2015).
\bibitem{kmm}
\textsc{D. Kalaj, M. Markovi\'c, M. Mateljevi\'c}:
\emph{Carath\'eodory and Smirnov type theorems for harmonic mappings of the unit disk onto surfaces.}
Ann. Acad. Sci. Fenn., Math. 38, No. 2, 565--580 (2013).

\bibitem{newj}
\textsc{D. Kalaj, E.  Saksman:}
\emph{Quasiconformal maps with controlled Laplacian.}
J. Anal. Math. \textbf{137}, No. 1, 251--268 (2019).



\bibitem{kalajarsen}
\textsc{D. Kalaj, A. Zlati\v canin}:
\emph{Quasiconformal mappings with controlled Laplacian and H\"older continuity.}
Ann. Acad. Sci. Fenn., Math. \textbf{44}, No. 2, 797--803 (2019).
\bibitem{kamz} \textsc{D. Kalaj:} {\it Quasiconformal harmonic mapping
between Jordan domains.}  Math. Z.  \textbf{260}, No. 2, 237--252,
2008.

\bibitem{kot}

\textsc{P. Koskela, J. Onninen, J. T.  Tyson:}
\emph{Quasihyperbolic boundary conditions and capacity: H\'older continuity of quasiconformal mappings.}
Comment. Math. Helv. \textbf{76}, No. 3, 416--435 (2001).

\bibitem{Les1}\textsc{F.  D. Lesley:}
\emph{H\"older Continuity of Conformal Mappings at the Boundary Via the Strip Method}.
Indiana University Mathematics Journal,
\textbf{31}, 341--354 (1982).
\bibitem{les}\textsc{F.  D. Lesley and S. E. Warschawski}:
\emph{On Conformal Mappings with Derivative in VMOA.} Math. Z. \textbf{158}, 275--283 (1978).

\bibitem{vesna}
\textsc{V. Manojlovi\'c}:
\emph{Bi-Lipschicity of quasiconformal harmonic mappings in the plane.}
Filomat \textbf{23}, No. 1, 85--89 (2009).
\bibitem{martio}
\textsc{O. Martio:}
\emph{On harmonic quasiconformal mappings.}
Ann. Acad. Sci. Fenn., Ser. A I \textbf{425},  (1968) 10 p.

\bibitem{mar1}
\textsc{O. Martio, R. N\"akki:} {\it Boundary H\"older continuity and
quasiconformal mappings.} J. London Math. Soc. (2) \textbf{44} (1991), no. 2,
339--350.

\bibitem{matvuo}
\textsc{M. Mateljevi\'c, M. Vuorinen:}
\emph{On harmonic quasiconformal quasi-isometries.}
J. Inequal. Appl. \textbf{2010}, Article ID 178732, 19 p. (2010).

\bibitem{naki}
\textsc{R. N\"akki, B.  Palka:}
\emph{Extremal length and H\"older continuity of conformal mappings.}
Comment. Math. Helv. \textbf{61}, 389--414 (1986).


\bibitem{zhu}
\textsc{D. Partyka, K.-I. Sakan, J.-F.  Zhu:}
\emph{Quasiconformal harmonic mappings with the convex holomorphic part.}
Ann. Acad. Sci. Fenn., Math. \textbf{43}, No. 1, 401--418 (2018); erratum ibid. 43, No. 2, 1085--1086 (2018).

\bibitem{MP} \textsc{M. Pavlovi\' c:}
{\it Boundary correspondence under harmonic quasiconformal
homeomorfisms of the unit disc}. Ann. Acad. Sci. Fenn., \textbf{27}, 365--372 (2002).

\bibitem{MP1} \textsc{M. Pavlovi\' c:}
{\it
Lipschitz conditions on the modulus of a harmonic function.}
Rev. Mat. Iberoam. \textbf{23}, No. 3, 831--845 (2007).

\bibitem{lib}
\textsc{M. Pavlovi\' c:}
\emph{Function classes on the unit disc. An introduction. 2nd revised and extended edition.}
De Gruyter Studies in Mathematics 52. Berlin: De Gruyter (ISBN 978-3-11-062844-9/hbk; 978-3-11-063085-5/ebook). xv, 553 p. (2019).

\bibitem{boundary} \textsc{C. Pommerenke:}
\emph{Boundary behaviour of conformal maps.}
Grundlehren der Mathematischen Wissenschaften. 299. Berlin: Springer-Verlag. ix, 300 p. (1992).
\bibitem{smirnov}
\textsc{V. I. Smirnov}:
\emph{\"Uber die R\"anderzuordnung bei konformer Abbildung. }
Math. Ann. \textbf{107}, 313--323 (1932).

\bibitem{wars}
\textsc{S. Warschawski:}
\emph{\"Uber einige Konvergenzs\"atze aus der Theorie der konformen Abbildung.}
Nachrichten G\"ottingen, \textbf{1930}, 344--369 (1930).

\bibitem{wars1}

\textsc{S. Warschawski:}
\emph{On conformal mapping of regions bounded by smooth curves.}
Proc. Am. Math. Soc. \textbf{2}, 254--261 (1951).

\bibitem{ZY}
\textsc{A. Zygmund:}  {\it Trigonometric Series I.} Cambrige University
Press, 1958.

\end{thebibliography}
\end{document}